\documentclass[smallextended,numbook,runningheads]{svjour3}     

\def\makeheadbox{\relax}   

\smartqed  
\usepackage{graphicx}
\usepackage{amsmath}
\usepackage{epstopdf} 
\usepackage{mathptmx}      
\usepackage{latexsym, amsfonts}
\newcommand{\RR}{\mathbb R}
\newcommand{\CC}{\mathbb C}

\newtheorem{defi}[theorem]{Definition}
\newtheorem{cor}[theorem]{Corollary}
\newtheorem{exm}[theorem]{Example}
\newtheorem{rem}[theorem]{Remark}

\begin{document}

\title{Inexact Arnoldi residual estimates and decay properties for functions of non-Hermitian matrices\thanks{This work has been supported by 
		      the FARB12SIMO grant, Universit\`a di Bologna, by INdAM-GNCS under the 2016 Project \emph{Equazioni e funzioni di matrici con struttura: analisi e algoritmi},
		      by the INdAM-GNCS ``Giovani ricercatori 2016'' grant, 
		      and by Charles University Research program No. UNCE/SCI/023.
}}


\author{Stefano Pozza         \and
        Valeria Simoncini 
}


\institute{S. Pozza \at
              Faculty of Mathematics and Physics, Charles University, 
			     Sokolovsk\'a 83, 186 75 \mbox{Praha 8}, Czech Republic,
			     and associated to ISTI-CNR, Pisa, Italy. \\
              Tel.: +420-951553365\\
              \email{pozza@karlin.mff.cuni.cz}           
           \and
           V. Simoncini \at
              Dipartimento di Matematica, Universit\`a di Bologna,
			     Piazza di Porta San Donato  5, I-40127 Bologna, Italy, 
			     and IMATI-CNR, Pavia. \\
			     \email{valeria.simoncini@unibo.it}   
}

\date{Submitted: \today}

\maketitle

\begin{abstract}
We derive a priori residual-type bounds for the Arnoldi approximation of a matrix function
  and a strategy for setting the iteration accuracies in the inexact Arnoldi approximation of matrix functions.
  Such results are based on the decay behavior of the entries of functions of banded matrices.
  Specifically, we will use a priori decay bounds for the entries of functions of banded non-Hermitian matrices by using Faber polynomial series.
  Numerical experiments illustrate the quality of the results.
\end{abstract}

\section{Introduction}
Matrix functions have arisen as a reliable and a computationally attractive tool for solving a large variety
of application problems; we refer the reader to \cite{higham} for a thorough discussion and references.
Given a  complex $n \times n$ matrix $A$ and a sufficiently regular function $f$,
we are interested in the approximation of the matrix function $f(A)$. 
More precisely, assuming $n$ large and $\mathbf v$ a unit norm vector, we want to approximate $f(A)\mathbf{v}$.
In this case, we can consider the orthogonal projection onto a subspace $\mathcal{V}_m$ of dimension $m$, obtaining the approximation
\begin{equation}\label{eq:arnoldi:approx}
 f(A)\,\mathbf{v} \approx V_m f(H_m) \, \mathbf{w},
\end{equation}
with $m$ much smaller than $n$, $V_m$ an $n \times m$ matrix whose columns are an orthogonal basis of $\mathcal{V}_m$, 
$H_m = V_m^* A V_m$, and $\mathbf{w} = V_m^* \mathbf{v}$.
In this paper, we will focus on the case in which $\mathcal{V}_m$ is the \emph{Krylov subspace}
 $$ \mathcal{K}_m(A,\mathbf{v}) = \textrm{span}\{\mathbf{v},A\mathbf{v},\dots,A^{m-1}\mathbf{v}\}$$
 and $V_m$ is the orthogonal basis obtained by the Arnoldi algorithm; see, e.g., \cite[chapter 13]{higham}. 
  Notice that the case of Arnoldi approximation for the matrix exponential has been especially considered. 
  Estimates of the error norm $\| e^{-tA} \mathbf{v} - V_m e^{-tH_m} \mathbf{e}_1 \|$
for $A$ non-normal have been given for instance
by Saad \cite{saad_1992}, by Lubich and Hochbruck in \cite{hochbruck_lubich},
  and recently by Wang and Ye in \cite{wang_ye_preprint} and \cite{wang_thesis}. 
  Other methods related to Arnoldi approximation can be found in \cite{AfaEieErnGut08,EieErnGut11,FroGutSch14_conv,FroGutSch14}
  where \emph{restarted} techniques are considered.
  Regarding rational Krylov approximations of matrix functions we refer the reader to the review \cite{guttel_2013}
  and to the black-box rational Arnoldi variant given in \cite{GutKni13}.

 When $V_n$ is the output of the Arnoldi algorithm, 
 $H_m$ is an upper Hessenberg matrix.
 Therefore the elements of $f(H_m)$ are usually characterized by a decay behavior.
 Indeed, given a square banded matrix $B$, the entries of the matrix function $f(B)$ 
 for a sufficiently regular function $f$ are characterized by a - typically exponential - decay pattern as they move away from the main diagonal. 
 This phenomenon has been known for a long time, and it is at the basis of approximations and estimation strategies in many fields, 
 from signal processing to quantum dynamics and multivariate statistics;
 see, e.g., \cite{benzi_boito_2014,benzi_boito_razouk_2013,benzi_simoncini} and their references.
 The interest in {\it a priori} estimates that can accurately predict the decay rate of matrix functions has significantly grown in the past decades, 
 and it has mainly focused on Hermitian matrices 
 \cite{demko_1977,eijkhout_polman_1988,meurant_review_1992,benzi_golub_1999,ye_2013,benzi_simoncini,delbuono_lopez_peluso_2005,canuto_simoncini_verani_2014};
the inverse and exponential functions have been given particular attention, due to their relevance in numerical analysis and other fields. 
Upper bounds usually take the form
\begin{eqnarray}\label{eqn:main}
|(f(B))_{k,\ell}| \le c \rho^{|k-\ell|},
\end{eqnarray}
where $\rho \in (0,1)$; both $\rho$ and $c$ depend on the spectral properties of $B$ and on the domain of $f$, 
while $\rho$ also strongly depends on the bandwidth of $B$.

In the case of a banded Hermitian matrix, 
bounds of the Arnoldi approximation have been used to obtain upper estimates for the entries decay of a related matrix function; 
see for instance \cite{benzi_simoncini} for the exponential function.
  Here we will exploit this connection but in the reverse direction. 
  More precisely, we will first derive decay bounds for the entries of banded non-Hermitian matrices.
  Then we will apply such bounds to the matrix function $f(H_m)$, with $H_m$ the upper Hessenberg matrix given by Arnoldi algorithm, 
  obtaining a priori bounds for a specifically defined residual associated with the approximation \eqref{eq:arnoldi:approx}; 
  these bounds complement those available in the already mentioned literature for the Arnoldi approximation.
  Furthermore, we will use the described bounds in the inexact Krylov approximation of matrix function; 
  in particular, the bounds can be used to devise a priori relaxing thresholds for the inexact matrix-vector multiplications with $A$, 
  whenever $A$ is not available explicitly. 
  These last results generalize the theory developed for $f(z)=z^{-1}$ and for the eigenvalue problem in \cite{Simoncini.Szyld.03} and \cite{Simoncini.05};
  see also \cite{DinSid17,KurFre18}.

The analysis of the decay pattern for banded {\it non-Hermitian} matrices is significantly harder compared to the Hermitian case, 
especially for non-normal matrices. 
In \cite{benzi_razouk} Benzi and Razouk addressed this challenging case for diagonalizable matrices. 
They developed a bound of the type (\ref{eqn:main}), where $c$ also contains the eigenvector matrix condition number.
In \cite{mastronardi_etall_2010} the authors derive several qualitative bounds, 
mostly under the assumption that $A$ is diagonally dominant. 
The exponential function provides a special setting, which has been explored in \cite{iserles_2000} 
and in \cite{wang_thesis,wang_ye_preprint}. 
In all these last articles, and also in our approach, bounds on the decay pattern
of banded non-Hermitian matrices are derived that avoid the explicit reference to the possibly large condition number of the eigenvector matrix. 
Specialized off-diagonal decay results have been obtained for certain normal matrices, see, e.g., \cite{Freund1989a,delbuono_lopez_peluso_2005,FroSchSch18}, 
and for analytic functions of banded matrices over $C^*$-algebras \cite{benzi_boito_2014}.

Starting with the pioneering work {\cite{demko_moss_smith}}, 
most estimates for the decay behavior of the entries have relied on Chebyshev and Faber polynomials as technical tool, mainly for two reasons. 
Firstly, polynomials of banded matrices are again banded matrices, although the bandwidth increases with the polynomial degree. 
Secondly, sufficiently regular matrix functions can be written in terms of Chebyshev and Faber series, 
whose polynomial truncations enjoy nice approximation properties for a large class of matrices, 
from which an accurate description of the matrix function entries can be deduced.
Using Faber polynomials we will present an original derivation of a family of bounds for function of banded non-Hermitian matrices.
Such family can be adapted to several cases, depending on function properties and matrix spectral properties.
Very similar bounds can be obtained combining Theorem 10 in \cite{benzi_boito_2014} with Theorem 3.7 in \cite{benzi_razouk}.
Another similar bound is given in \cite[Theorem 2.6]{mastronardi_etall_2010} for the case of multi-banded matrices
and in \cite[Theorem 3.8]{wang_thesis} for the exponential case.
See also \cite{PozTud18} where the bound we present here have been extended to matrices with a more general sparsity pattern.
The bound we will present and the ones just cited above are based on the approximation of the field of value (numerical range) of a matrix,
which is in general expensive to compute.
Nevertheless, it is not necessary to have a precise approximation of the field of value in order to use such bounds.
Moreover, in several cases an approximation of the field of value can be obtained more easily, see, e.g., \cite{eiermann93} (in particular section 3 for for Toeplitz matrices), 
and \cite[Section 5.3]{PozTud18} for the adjacency matrix of a network.

 The paper is organized as follows.
In section \ref{sec-faber-poly}
we use Faber polynomials to give a bound that can be adapted 
to approximate the entries of several functions of banded matrices; as an example we consider
the functions $e^A$ and $e^{-\sqrt{A}}$.
In section~\ref{subsec-arnoldi} we first show that the derived bounds
can be used for a residual-type bound  in the approximation of $f(A)\mathbf{v}$,
for certain functions $f$  by means of the Arnoldi algorithm. Then 
we describe how to employ this bound to reliably estimate the quality of the approximation
when in the Arnoldi iteration the accuracy in the matrix-vector product is relaxed.
Numerical experiments illustrate the quality of the bounds.
We conclude with some remarks in section \ref{sec-conclusion}
and with technical proofs in the appendix.

 All our numerical experiments were performed using Matlab (R2013b) \cite{matlab7}.
  In all our experiments, the computation of the field of values 
  employed the code in \cite{cowen_harel_1995}.

\section{Decay bounds for functions of banded matrices}\label{sec-faber-poly}  
 We begin recalling the definition of matrix function and some of its properties.
   Matrix functions can be defined in several ways (see \cite[section 1]{higham}).   
 For our presentation it is helpful to introduce the definition  that employs the Cauchy integral formula.
   \begin{defi}\label{def-matrix-function}
Let $A\in\CC^{n\times n}$ and $f$ be an analytic function on some open $\Omega \subset \mathbb{C}$.
     Then
$$ 
f(A) = \int_\Gamma f(z) \left (zI - A  \right )^{-1} \textrm{d}z, 
$$
      with  $\Gamma \subset \Omega$ a system of Jordan curves encircling each eigenvalue of $A$
      exactly once, with mathematical positive orientation.
   \end{defi}
    
   When $f$ is analytic Definition \ref{def-matrix-function} is equivalent 
   to other common definitions; see \cite[section 2.3]{rinehart}.
   
For $\mathbf{v} \in \mathbb{C}^{n}$ we denote with $|| \mathbf{v} ||$ the Euclidean vector norm, 
and for any matrix $A \in \mathbb{C}^{n \times n}$, with $|| A ||$ the induced matrix norm, that is
 $||A||= \sup_{||\mathbf{v}||=1} ||A \mathbf{v}||$. $\CC^+$ denotes the open right-half complex plane.
Moreover, we recall that the \emph{field of values} (or \emph{numerical range}) of $A$ is defined as
the set $W(A) = \{ \mathbf{v}^* A \mathbf{v} \, | \, \mathbf{v}\in\mathbb{C}^{n},  ||\mathbf{v}|| = 1 \}$,
where ${\mathbf v}^*$ is the conjugate transpose of ${\mathbf v}$.
We remark that the field of values of a matrix is a bounded convex subset of $\mathbb{C}$.

The $(k,\ell)$ element of a matrix $A$ will be denoted by $(A)_{k,\ell}$.
The set of banded matrices is defined as follows.
 \begin{defi}
   The notation $\mathcal{B}_n(\beta,\gamma)$ defines the set of banded matrices $A \in \mathbb{C}^{n \times n}$
   with upper bandwidth $\beta \geq 0$ and lower bandwidth $\gamma \geq 0$, i.e.,
   $(A)_{k,\ell} = 0$ for $\ell - k > \beta$ or $k - \ell > \gamma$.
 \end{defi}
 
 We observe that if $A \in {\mathcal B}_n(\beta,\gamma)$ with $\beta, \gamma \ne 0$, for
 \begin{equation}\label{eq-xi-def}
      \xi := \left \{ \begin{array}{lc}
		      \lceil (\ell - k) / \beta \rceil, & \textrm{ if } k < \ell \\
		      \lceil (k - \ell) / \gamma \rceil, & \textrm{ if } k \ge \ell 
                  \end{array} \right. 
 \end{equation}
it holds that
 \begin{equation}\label{eq-power-A-banded}
 (A^m)_{k,\ell} = 0, \quad \textrm{ for every } m < \xi.
\end{equation}
This characterization of banded matrices is a classical fundamental tool to prove the decay property
of matrix functions, as sufficiently regular functions can be expanded in power series.
Since we are interested in nontrivial banded matrices, in the following we shall
assume that both $\beta$ and $\gamma$ are nonzero.

Faber polynomials extend the theory of power series to sets different from the disk,
and can be effectively used to bound the entries of matrix functions. 
Let $E$ be a continuum (i.e., a non-empty, compact and connected subset of $\mathbb{C}$)
 with connected complement, then by Riemann's mapping theorem there exists
 a function $\phi$ that maps the exterior of $E$ conformally onto the exterior
 of the unitary disk $\{|z|\leq 1\}$ and so that
$$ 
\phi(\infty) = \infty, \quad \lim_{z \rightarrow \infty} \frac{\phi(z)}{z} =  d > 0.
$$
Hence, $\phi$ can be expressed by a Laurent expansion
$\phi(z) = dz + a_0 + \frac{a_1}{z} + \frac{a_2}{z^2} + \cdots$. 
Furthermore, for every $n >0$ we have
$$
 \left ( \phi(z) \right)^n = 
    d z^n + a_{n-1}^{(n)} z^{n-1} + \dots + a_{0}^{(n)} + \frac{a_{-1}^{(n)}}{z} + \frac{a_{-2}^{(n)}}{z^2} + \cdots .
$$
Then, the Faber polynomial for the domain $E$ is defined by
(see, e.g., \cite{suetin})
$$ 
\Phi_n(z) = d z^n + a_{n-1}^{(n)} z^{n-1} + \dots + a_{0}^{(n)}, \quad \textrm{ for } n \geq 0. 
$$
If $f$ is analytic on $E$ then it can be expanded in a series
of Faber polynomials for $E$, that is
$$
f(z) = \sum_{j = 0}^\infty f_j \Phi_j(z), \quad \textrm{ for } z \in E;
$$
\cite[Theorem 2, p.~52]{suetin}.
If the spectrum of $A$ is contained in $E$ and $f$ is a function analytic on $E$, then
the matrix function $f(A)$ can be expanded as follows
(see, e.g., \cite[p.~272]{suetin})
\begin{equation*}
f(A) = \sum_{j = 0}^\infty f_j \Phi_j(A).
\end{equation*}
If, in addition, $E$ contains the field of values $W(A)$, 
then for $n\geq1$ we get
\begin{equation}\label{eq:beck:ineq}
   \| \Phi_n(A) \| \leq 2,
\end{equation}
by Beckermann's Theorem 1.1 in \cite{beckermann_2005}.

By using the properties of Faber polynomials, 
in the following theorem we will derive decay bounds for a large class of matrix functions.
Notice that the estimate in \cite[Theorem 10]{benzi_boito_2014} 
combined with the results presented in \cite[Theorem 3.7]{benzi_razouk} results in similar bounds
(see also \cite{ellacott_1983});
moreover, in section 2 of \cite{mastronardi_etall_2010}, and in particular in Theorem 2.6, analogous results are discussed.
Another similar bound can be found in \cite[Theorem 3.8]{wang_thesis} for the exponential case.
The derivation we will describe differs from the ones listed above by using inequality \eqref{eq:beck:ineq}.

\begin{theorem}\label{thm-decay-level-set}
   Let $A \in \mathcal{B}_n(\beta,\gamma)$ with field of values contained in a convex continuum $E$. 
   Moreover, let $\phi$ be the conformal map sending the exterior of $E$ onto the exterior of the unitary disk,
   and let $\psi$ be its inverse.
   For any $\tau > 1$ so that $f$ is analytic on the level set $G_\tau$ 
   defined as the complement of the set $\{ \psi(z) \,\, : \,\, |z| > \tau\}$, we get
   $$ 
      \left| \left ( f(A) \right )_{k,\ell} \right|	\leq  
      2 \frac{\tau}{\tau - 1} \max_{|z| = \tau} \left| f(\psi(z)) \right| \left ( \frac{1}{\tau} \right)^\xi ,
   $$
   with $\xi$ defined by (\ref{eq-xi-def}).
\end{theorem}
\begin{proof}
 By Properties \eqref{eq-power-A-banded} and \eqref{eq:beck:ineq} we get 
\begin{eqnarray*}
   |(f(A))_{k,\ell}| = \left |\sum_{j = 0}^{\infty} f_j \left( \Phi_j (A) \right)_{k,\ell} \right | 
		     = \left |\sum_{j = \xi}^{\infty} f_j \left( \Phi_j (A) \right)_{k,\ell} \right|
		     \leq 2 \sum_{j = \xi}^\infty |f_j|,
\end{eqnarray*}
 where the Faber coefficients $f_j$ are given by
(see, e.g., \cite[chapter III,Theorem 1]{suetin})
$$
 f_j = \frac{1}{2 \pi i} \int_{|z| = \tau} \frac{f(\psi(z))}{(z)^{j + 1}} \, \textrm{d}z.
$$
Since
$|f_j| \leq \frac{1}{(\tau)^{j}} \max_{|z| = \tau} \left| f(\psi(z)) \right|$
we get
\begin{eqnarray*}
      \left| \left ( f(A) \right )_{k,\ell} \right| 
	     &\leq&	
2 \, \max_{|z| = \tau} \left| f(\psi(z)) \right| \sum_{j = \xi}^\infty \left ( \frac{1}{\tau} \right)^j 
			      = 
2 \, \frac{\tau}{\tau - 1} \max_{|z| = \tau} \left| f(\psi(z)) \right| \left ( \frac{1}{\tau} \right)^\xi  . 
\end{eqnarray*} \qed
\end{proof}

The choice of $\tau$ in Theorem~\ref{thm-decay-level-set}, and thus the sharpness of the derived estimate,  
depends on the trade-off between the possible large size of $f$ on the given region, 
 and the exponential decay of $(1/\tau)^\xi$, and thus it produces an infinite family of bounds depending on the problem considered.
In our examples, 
we will apply Theorem~\ref{thm-decay-level-set} to the approximation of the functions: $f(z)=e^z$ 
and $f(z)=e^{-\sqrt{z}}$, with $z$ in a properly chosen domain.
\begin{cor}\label{cor_exp}
   Let $A \in \mathcal{B}_n(\beta,\gamma)$  with field of values contained in a closed set $E$ 
whose boundary is a horizontal ellipse with semi-axes  $a \geq b > 0$ and center $c = c_1 + i c_2 \in \CC$, $c_1, c_2 \in \RR$.
   Then 
   $$ 
   \left| \left ( e^{A} \right )_{k,\ell} \right|	\leq  
 		2 e^{c_1}  \frac{\xi + \sqrt{\xi^2 + a^2 - b^2}}{\xi + \sqrt{\xi^2 + a^2 - b^2} - (a+b)} 
			\left( \frac{a+b}{\xi}\frac{e^{q(\xi)}}{1 + \sqrt{1 + (a^2 - b^2)/\xi^2}} \right)^\xi, 
   $$ 
   for $\xi > b$, with
   $q(\xi) = 1 + \frac{a^2 - b^2}{\xi^2 + \xi\sqrt{\xi^2 + a^2 - b^2}}$
   and $\xi$ as in (\ref{eq-xi-def}).
\end{cor}

\noindent The proof can be found in the appendix.
Notice that for $\xi$ large enough, the decay rate is of the form $( (a+b)/(2\xi))^\xi$, that is, the decay is super-exponential.
Moreover, in the Hermitian case 
we can let $b \rightarrow 0$ in Corollary \ref{cor_exp}, thus obtaining a bound with a similar decay rate to the one derived in \cite{benzi_simoncini}.

The function $f(z)=e^{-\sqrt{z}}$ is not analytic on the whole complex plane.
This property has crucial effects in the approximation.
\begin{cor}\label{cor_expsqrt}
   Let $A \in \mathcal{B}_n(\beta,\gamma)$  with 
field of values contained in a closed set $E \subset \CC^+$,
whose boundary is a horizontal ellipse
   with semi-axes  $a \geq b > 0$ and center $c \in \CC$.
  Then,
$$ 
\left|  \left ( e^{-\sqrt{A}} \right )_{k,\ell} \right|	\leq  
		  2 q_2(a,b,c)  
		  \left (  \frac{a+b}{|c|} \, \frac{1}{|1 + \sqrt{1 - (a^2 - b^2)/c^2}|}  \right )^\xi,
$$
   with $\xi$ defined by (\ref{eq-xi-def}) and  
$$ 
q_2(a,b,c) = \frac{ \left| c + \sqrt{c^2 -(a^2 - b^2)}\right|} { \left|c + \sqrt{c^2 -(a^2 - b^2)} \right| - (a+b)} . 
$$
\end{cor}

\noindent The proof is given in the appendix.
Notice that when $c$ is not real (e.g., when $A$ is not real),
then the bound in Corollary \ref{cor_expsqrt} can be further improved since the ellipses considered in the proof are not the maximal one.

\begin{rem}
 For the sake of simplicity in the previous corollaries 
 horizontal ellipses were employed.
 However, more general convex sets $E$ may be considered.
 The previous bounds will change accordingly, since
 the optimal value for $\tau$ in Theorem~\ref{thm-decay-level-set}
 does depend on the parameters associated with $E$.
 For instance, for the exponential function and a \emph{vertical} ellipse, 
we can derive the same bound as in Corollary~\ref{cor_exp}
by letting $b > a$ (notice that this is different 
from exchanging the role of $a$ and $b$ in the bound).
The proof of this fact is non-trivial but technical, and it is not reported.
\end{rem}

\section{Residual bounds for exact and inexact Arnoldi methods}\label{subsec-arnoldi}
Given a matrix $A \in \mathbb{C}^{n\times n}$ and a vector $\mathbf{v} \in \mathbb{C}^n$, 
then for $m\ge 1$, the $m$th step of the Arnoldi algorithm determines an orthonormal basis $\{\mathbf{v}_1,\dots,\mathbf{v}_{m}\}$
for the Krylov subspace $\mathcal{K}_m(A,\mathbf{v})$, 
the subsequent orthonormal basis vector $\mathbf{v}_{m+1}$, 
an $m \times m$ upper Hessenberg matrix $H_m$, and a scalar $h_{m+1,m}$ such that
$$ 
A V_m = V_m H_m + h_{m+1,m} \mathbf{v}_{m+1} \mathbf{e}_m^T ,
$$
where $V_m = [\mathbf{v}_1,\dots,\mathbf{v}_{m}]$.
  Due to the orthogonality of the columns of $[V_m,v_{m+1}]$, 
the matrix $H_m$ is the projection and restriction
of $A$ onto $\mathcal{K}_m(A,\mathbf{v})$, that is $ H_m=V_m^* A V_m$.
Assuming, without loss of generality, that $\| \mathbf{v} \| = 1$,
the Arnoldi approximation to $f(A)\mathbf{v}$ is given as $V_m f(H_m) \mathbf{e}_1$; 
see, e.g., \cite[chapter 13]{higham}. 
The quantity 
$$|{\mathbf e}_m^T f(H_m){\mathbf e}_1|$$ is commonly 
used to monitor the accuracy of the approximation $\|f(A){\mathbf v} - V_m f(H_m){\mathbf e}_1\|$.
Notice that $|{\mathbf e}_m^T f(H_m){\mathbf e}_1| = |(f(H_m))_{m,1}|$, the last entry of the
first column of $f(H_m)$. In the case of the
exponential, $e^{-tA}{\mathbf v}$,  the quantity 
$$r_m(t) = | h_{m+1,m} {\mathbf e}_m^T e^{-tH_m}{\mathbf e}_1|$$
can be interpreted as the ``residual'' norm of an associated differential
equation, see \cite{Botchev.Grimm.Hochbruck.13} and references therein; this is true also for other
functions, see, e.g., \cite[section 6]{Druskin.Knizhnerman.95}. 
Indeed, assume that ${\mathbf y}(t)=f(tA){\mathbf v}$
is the solution to the differential equation $y^{(d)} = A y$ for some $d$th derivative,
 $d\in{\mathbb N}$ and
 specified initial conditions for $t=0$. Let
 ${\mathbf y}_m(t) = V_m f(tH_m){\mathbf e}_1 =: V_m \widehat {\mathbf y}_m(t)$. The vector $\widehat {\mathbf y}_m(t)$
is the solution to the projected equation $\widehat {\mathbf y}_m^{(d)} = H_m \widehat {\mathbf y}_m$ with initial
condition $\widehat {\mathbf y}_m(0)={\mathbf e}_1$.
The differential equation residual ${\mathbf r}_m = A {\mathbf y}_m - {\mathbf y}_m^{(d)}$ can be used to monitor the accuracy of the approximate solution. 
Indeed, using the definition of ${\mathbf y}_m$ and the Arnoldi relation, we get
\begin{eqnarray*}
{\mathbf r}_m(t) &=& A {\mathbf y}_m - {\mathbf y}_m^{(d)} = A V_m f(tH_m){\mathbf e}_1 - {\mathbf y}_m^{(d)} \\
&=&
V_m H_m f(tH_m){\mathbf e}_1 - V_m (f(tH_m))^{(d)} {\mathbf e}_1 + 
{\mathbf v}_{m+1} h_{m+1,m} {\mathbf e}_m^Tf(tH_m){\mathbf e}_1 \\
&=& V_m ( H_m \widehat{\mathbf y}_m -  \widehat{\mathbf y}_m^{(d)} )  + {\mathbf v}_{m+1} h_{m+1,m} {\mathbf e}_m^Tf(tH_m){\mathbf e}_1  \\
&=&
{\mathbf v}_{m+1} h_{m+1,m} {\mathbf e}_m^Tf(tH_m){\mathbf e}_1.
\end{eqnarray*}
Therefore $ r_m(t) = \| {\mathbf r}_m(t)\|$.

Without loss of generality in the following we consider $t=1$.
Hence, for simplicity, we will denote $ r_m = r_m(1)$, and $ {\mathbf r}_m = {\mathbf r}_m(1)$.
We remark that the property $H_m=V_m^* AV_m$ ensures that the field of values
of $H_m$ is contained in that of $A$, so that our theory can be applied
using $A$ as reference matrix to individuate the spectral region of interest.
Let $a$, $b$ be the semi-axes and $c=c_1+i c_2$ the center of an elliptical region $E$ containing the field of values of $A$
  {and $\xi = m - 1$.} From Corollary~\ref{cor_exp} for $m >  b + 1$ we get the inequality
  \begin{equation}\label{eq-arnoldi-prebound}
    |r_m| \leq  h_{m+1,m}
			2 e^{-c_1} p(m)  
			\left( \frac{e^{q(m-1)}(a+b)}{m-1 + \sqrt{(m-1)^2 + (a^2 - b^2)}} \right)^{m-1},
  \end{equation}
   with 
   $$q(m-1) = 1 + \frac{(a^2 - b^2)}{(m-1)^2 + (m - 1)\sqrt{(m-1)^2 + (a^2 - b^2)}}$$
   and 
   $$p(m) = \frac{m-1 + \sqrt{(m-1)^2 + (a^2 - b^2)}}{m-1 + \sqrt{(m-1)^2 + (a^2 - b^2)} - (a+b)}.$$
  In \cite{wang_thesis,wang_ye_preprint} a similar bound is proposed, where however a continuum $E$ 
  with rectangular shape is considered, instead of the elliptical one we take in Corollary~\ref{cor_exp}.
  Experiments suggest that the sharpness of these bounds depends on which set $E$ 
  better approximates the matrix field of values.

\begin{figure}[tbp]
\centering
    \includegraphics[width = 0.49\textwidth]{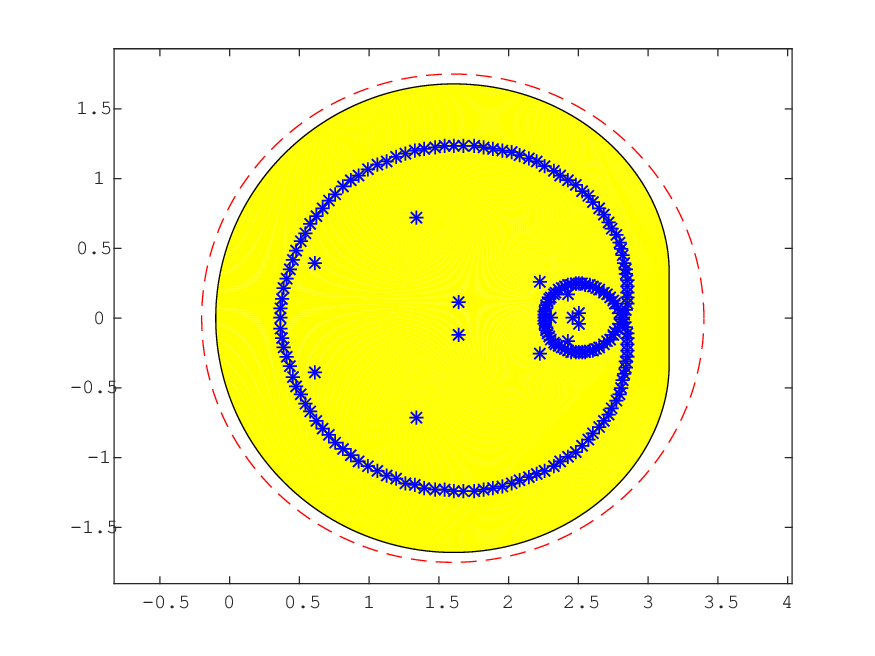}
    \includegraphics[width = 0.49\textwidth]{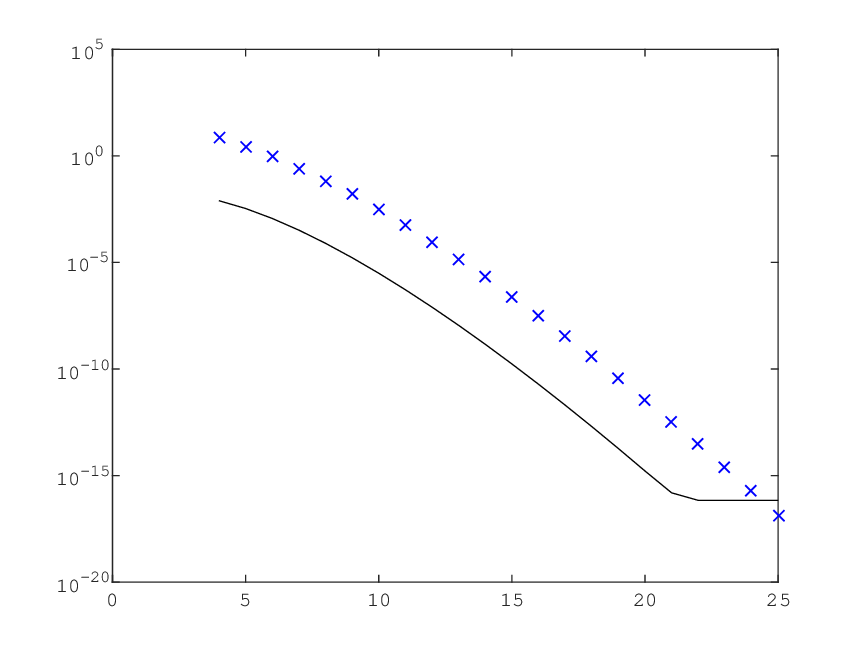}
    \includegraphics[width = 0.49\textwidth]{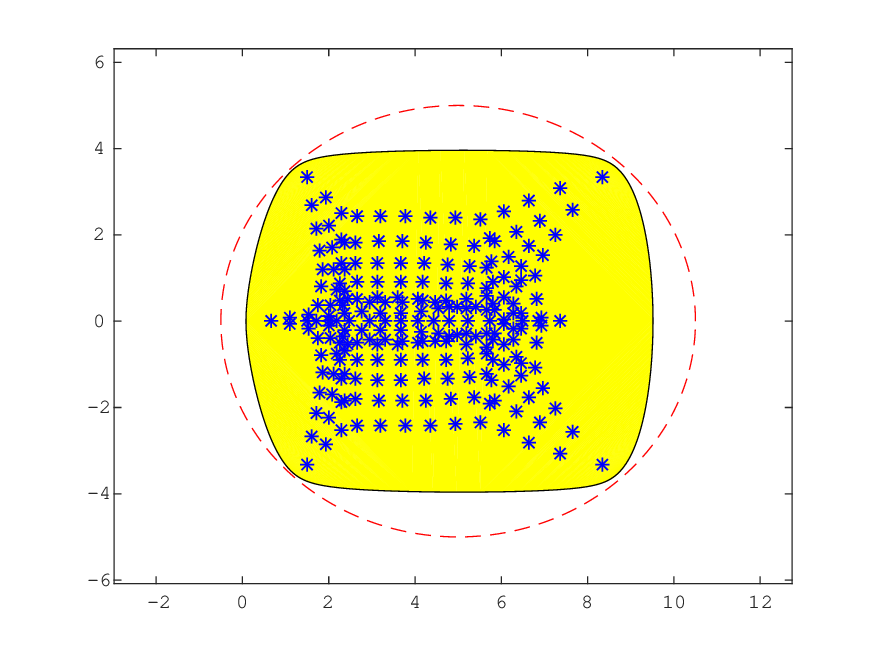}
    \includegraphics[width = 0.49\textwidth]{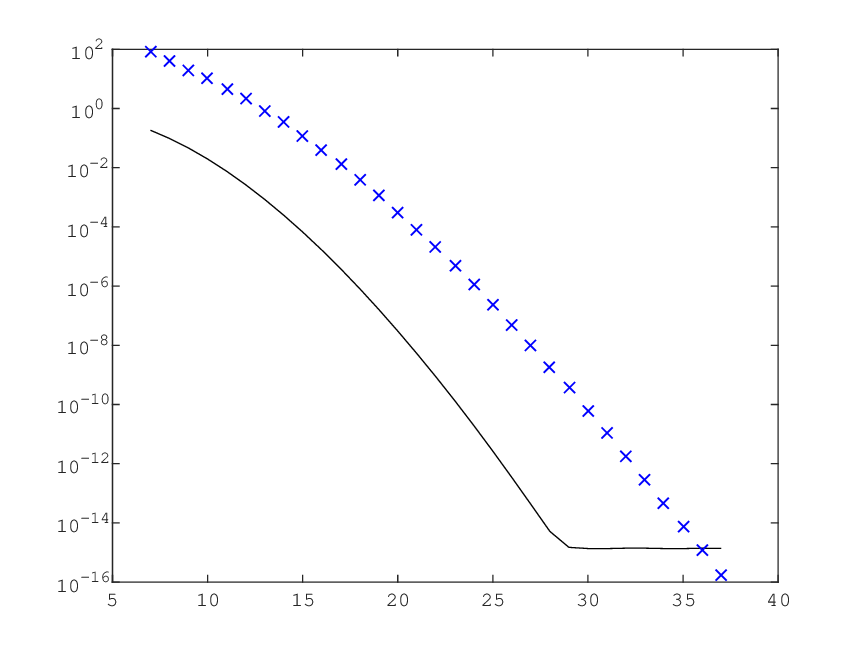}
\caption{Example \ref{ex-arnoldi}. Approximation of $e^{-A}{\mathbf{v}}$, 
with $\mathbf{v} = (1,\dots,1)^T / \sqrt{n}$.
Top: $A =\textrm{Toeplitz}(-1, 1, \underline{2}, 0.1) \in {\cal B}_{200}(1,2)$.
Bottom: matrix {\tt pde225}. 
Left: $W(A)$ (yellow area), 
eigenvalues of $A$ (blue crosses), 
and enclosing ellipse $E$ (red dashed line). 
Right: residual norm as the Arnoldi iteration proceeds in the approximation
(black solid line),
and residual bound in (\ref{eq-arnoldi-prebound}) (blue crosses).
\label{fig-sec3-arnoldi}
}
\end{figure}
  
\begin{exm}\label{ex-arnoldi}
  Figure \ref{fig-sec3-arnoldi} shows the behavior of the bound in (\ref{eq-arnoldi-prebound}) 
  for the {residual} of the Arnoldi approximation of $e^{-A}\mathbf{v}$ with $\mathbf{v}=(1,\dots,1)^T / \sqrt{n}$.
  The top plots refer to $A  \in {\cal B}_{200}(1,2)$ with Toeplitz structure, $A =\textrm{Toeplitz}(-1, 1, \underline{2}, 0.1)$,
  where the underlined element is on the diagonal, while the previous (resp. subsequent) values denote the lower (resp. upper) diagonal entries.
  The bottom plots refer to the matrix pde225 of the Matrix Market repository \cite{matrixmarket}.
  The left figure shows the field of values of the matrix $A$ (yellow area), 
  its eigenvalues (``$\times$''), and the horizontal ellipse used in the bound (red dashed line).
  On the right we plot the residual associated with the Arnoldi approximation as the iteration
proceeds (black solid line), and the corresponding values of the bound (blue crosses).
  Matrix exponentials were compute by the \verb expm  Matlab function.
\end{exm}	

\vskip 0.1in \sloppy
In an inexact Arnoldi procedure $A$ is not known exactly. This may be due for instance to the
fact that $A$ is only implicitly available via functional operations with a vector, which can be approximated
at some accuracy. To proceed with our analysis we can formalize this inexactness at each iteration $k$ as
 \begin{eqnarray}\label{eqn:inexact}
\widetilde {\mathbf v}_{k+1} = A {\mathbf v}_k + {\mathbf w}_k \approx A {\mathbf v}_k.
\end{eqnarray}
Typically, some form
of accuracy criterion is implemented, so that $\|{\mathbf w}_k\| < \epsilon$ for some $\epsilon$. It may be
that a different value of this tolerance is used at each iteration $k$, so that $\epsilon=\epsilon_k$.  
The new vector
$\widetilde {\mathbf v}_{k+1}$ is then orthonormalized with respect to the previous basis vectors to
obtain ${\mathbf v}_{k+1}$. In compact form, the original Arnoldi relation becomes
$$
(A+{\cal E}_m) V_m = V_m H_m + h_{m+1,m} \mathbf{v}_{m+1} \mathbf{e}_m^T, \quad 
{\cal E}_m = [{\mathbf w}_1, \ldots, {\mathbf w}_m] V_m^* .
$$
Here $H_m$ is again upper Hessenberg; however, $H_m = V_m^* (A+{\cal E}_m) V_m$.
Moreover, ${\cal E}_m$ changes as $m$ grows.
The differential equation residual can be defined in the same way as for the
exact case, ${\mathbf r}_m = A {\mathbf y}_m - {\mathbf y}_m^{(d)}$; however the 
inexact Arnoldi relation should be considered to proceed further. Indeed,
\begin{eqnarray*}
{\mathbf r}_m &=& A {\mathbf y}_m - {\mathbf y}_m^{(d)} = A V_m f(H_m){\mathbf e}_1 - {\mathbf y}_m^{(d)} \\
&=&- {\cal E}_m V_m f(H_m){\mathbf e}_1 + V_m H_m f(H_m){\mathbf e}_1 - {\mathbf y}_m^{(d)} + 
{\mathbf v}_{m+1} h_{m+1,m} {\mathbf e}_m^Tf(H_m){\mathbf e}_1  \\
&=&
- [{\mathbf w}_1, \ldots, {\mathbf w}_m]  f(H_m){\mathbf e}_1 + 
{\mathbf v}_{m+1} h_{m+1,m} {\mathbf e}_m^Tf(H_m){\mathbf e}_1.
\end{eqnarray*}
Note that $\|{\mathbf r}_m\|$ is not available, since $A$ cannot be applied exactly.
However, with the previous notation we can write
$ \|{\mathbf r}_m\|   \le |\|{\mathbf r}_m\| - r_m | + r_m$ where
$$
|\|{\mathbf r}_m\| - r_m | \le \|[{\mathbf w}_1, \ldots, {\mathbf w}_m]  f(H_m){\mathbf e}_1\|;
$$
we remark that in this case $r_m \neq \| {\mathbf r}_m \|$.
Therefore, checking the available $r_m$ provides a good measure of the accuracy in
the function estimation as long as $\|[{\mathbf w}_1, \ldots, {\mathbf w}_m]  f(H_m){\mathbf e}_1\|$
is smaller than the requested tolerance for the final accuracy of the computation.

Clearly, $\|[{\mathbf w}_1, \ldots, {\mathbf w}_m]  f(H_m){\mathbf e}_1\| \le 
\|[{\mathbf w}_1, \ldots, {\mathbf w}_m]\|\, \|f(H_m){\mathbf e}_1\|$ so that the criterion
$\|{\mathbf w}_k\| < \epsilon$ can be used to monitor the quality of the approximation to 
$f(A){\mathbf v}$ by means of $r_m$.
However, a less stringent criterion can be devised.
Following similar discussions in \cite{Simoncini.Szyld.03},\cite{Simoncini.05}, we write
\begin{equation*}
 \|[{\mathbf w}_1, \ldots, {\mathbf w}_m]  f(H_m){\mathbf e}_1\| =
\|\sum_{j=1}^m {\mathbf w}_j {\mathbf e}_j^T f(H_m){\mathbf e}_1 \| \le \sum_{j=1}^m \|{\mathbf w}_j\|\,
|{\mathbf e}_j^T f(H_m){\mathbf e}_1|,
\end{equation*}
where we assume
that $\|{\mathbf w}_j\| < \epsilon_j$, that is the accuracy in the computation with $A$
varies with $j$.
Hence, $\|[{\mathbf w}_1, \ldots, {\mathbf w}_m]  f(H_m){\mathbf e}_1\|$ is small
when either $\|{\mathbf w}_j\|$ or $|{\mathbf e}_j^T f(H_m){\mathbf e}_1|$ is small, and not
necessarily both. 
 By exploiting the exponential decay of the entries of $f(H_m){\mathbf e}_1$, we can
infer that $\|{\mathbf w}_j\|$ is in fact allowed to grow with $j$, according with the
exponential decay of the corresponding entries of $f(H_m){\mathbf e}_1$, without
affecting the overall accuracy. A priori bounds on
$|{\mathbf e}_j^T f(H_m){\mathbf e}_1|$ can be used to select $\epsilon_j$ when
estimating $A {\mathbf v}_j$. This relaxed strategy can significantly decrease the
computational cost of matrix function evaluations whenever applying $A$ accurately is expensive.
However, notice that the field of values of $H_m$ is contained in the field of values of $A +{\cal E}_m$.
Hence if $W(A)$ is contained in an ellipse {$\partial E$} of semi-axes $a,b$ and center $c$ then
$W(A + {\cal E}_m) \subset W(A) + W({\cal E}_m)$.
Since 
$$
\sup_{\|z\|=1} | z^*{\cal E}_m z | \leq \sup_{\|z\|=1} \| {\cal E}_m z \| 
\leq \sqrt{\sum_{j=1}^m \|{\mathbf w}_j\|^2} \leq \sqrt{\sum_{j=1}^m \epsilon_j^2} =: \epsilon^{(m)},
$$
the set $W({\cal E}_m)$ is contained in the disk centered at the origin
and radius $\epsilon^{(m)}$.  Therefore, $W(A) + W({\cal E}_m)$ is contained 
in any set whose boundary has minimal distance from $\partial E$ not smaller than $\epsilon^{(m)}$.
One such set is contained in the ellipse $\partial E_m$
with semi-axes $a(1 + \epsilon^{(m)}/b)$, $b + \epsilon^{(m)}$ and center $c$. Indeed,
$z\in \partial E_m$ can be parameterized as
$$
z= \left(1 + \frac{\epsilon^{(m)}}{b}\right)\frac{\rho}{2}\left(R e^{i\theta} +
 \frac{1}{Re^{i\theta}}\right) + c, \quad 0\leq \theta \leq 2\pi,
$$
with $\rho = \sqrt{a^2 - b^2}$, $R = (a + b)/\rho$.
The distance between $z$ and the ellipse $\partial E$ is
$$
\left| \frac{\epsilon^{(m)}}{b}\frac{\rho}{2}\left(R e^{i\theta} + \frac{1}{Re^{i\theta}}\right) \right|
\geq \left| \frac{\epsilon^{(m)}}{b}\frac{\rho}{2}\left(R - \frac{1}{R}\right) \right|
= \epsilon^{(m)}.
$$ 
Hence we can provide the following strategy for the choice of the accuracy in inexact Arnoldi.
\begin{theorem}
Let ${\mathbf r}_j = A {\mathbf y}_j - {\mathbf y}_j^{(d)}$ be the residual obtained by the $j$th step of the inexact Arnoldi algorithm, 
with accuracy $ \|w_j\| \leq \bar \epsilon_j$, for $j=1,2,\dots$.
Consider an ellipse with semiaxes $a \geq b > 0$ and center $c$ containing $W(A)$.
Moreover, let us fix a tolerance $tol>0$, a  maximum number of iterations $m$ and a value $\epsilon^{(m)} > 0$.
Then, the following choice for the accuracies
\begin{equation}\label{eq.choice.accuracy}
  \bar \epsilon_j = \left \{
   \begin{array}{ll}
    \dfrac{tol}{m}\max\left(1,\dfrac{1}{s_j}\right),       & \textrm{ if } \quad \dfrac{tol}{m  \, s_j} < \dfrac{\sqrt{ (\epsilon^{(m)})^2 - \sum_{k=1}^{j-1}\bar \epsilon_k^2}} {m - j + 1}  \\ \\
     \dfrac{\sqrt{(\epsilon^{(m)})^2 - \sum_{k=1}^{j-1} \bar \epsilon_k^2 }} {m - j + 1},    & \textrm{ otherwise } 
   \end{array}
   \right.
\end{equation}
for $j=1,\dots,m$, gives 
$$\sqrt{\sum_{j=1}^m \bar \epsilon_j^2} \leq \epsilon^{(m)}, \quad \textrm{ and } \quad |\|{\mathbf r}_m\| - r_m | \le tol,$$
where $s_j$ is the upper bound for $|{\mathbf e}_j^T f(H_m){\mathbf e}_1|$  from Theorem \ref{thm-decay-level-set}
with $f$ the function associated with the solution of the differential equation $y^{(d)} = A y$,
and $E$ the ellipse with semiaxes $a(1 + \epsilon^{(m)}/b)$, $b + \epsilon^{(m)}$ and center $c$.
The bound can be specialized for the functions $f(z)=e^z$ and $f(z)=e^{-\sqrt{z}}$ 
using respectively corollaries \ref{cor_exp} and \ref{cor_expsqrt}.
\end{theorem}

\begin{figure}[tbp]
\centering
    \includegraphics[width = 0.49\textwidth]{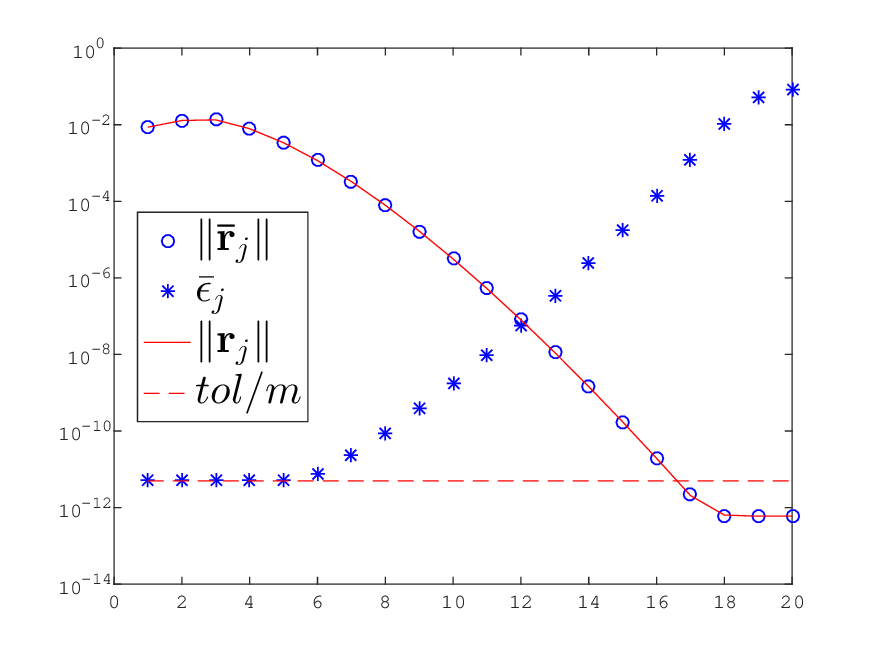}
    \includegraphics[width = 0.49\textwidth]{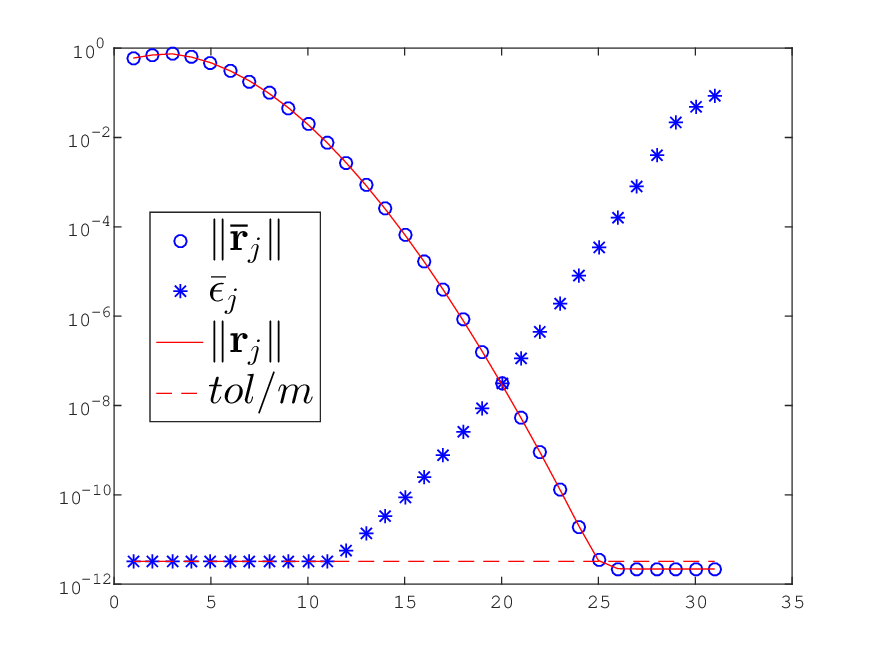}
\caption{Example \ref{ex-inex}, approximation of $e^{-A} \mathbf{v}$
with $\mathbf{v} = (1,\dots,1)^T / \sqrt{n}$. 
Residual norm $\|\mathbf{r}_j\|$ with constant accuracy $\epsilon_j = tol/m$,
and residual norm $\|\mathbf{\bar r}_j\|$ with  $\epsilon_j = \bar \epsilon_j$ by (\ref{eq.choice.accuracy})
as the inexact Arnoldi method proceeds.
Left: For $A =\textrm{Toeplitz}(1,\underline{2}, 0.1, -1) \in {\cal B}_{200}(1,1)$.
Right: For matrix {\tt pde225} from the Matrix Market repository \cite{matrixmarket}.
\label{fig.inexact.arnoldi}
}
\end{figure}

\begin{exm}\label{ex-inex}
We consider the inexact Arnoldi procedure for the approximation of $\exp(-A)\mathbf{v}$,
so that  the norm of the differential equation residual 
is lower than a tolerance $tol$.  
The inexact matrix-vector product was implemented as in (\ref{eqn:inexact}),
where ${\bf w}_j$ is a random vector of norm $\epsilon_j$. 
Figure \ref{fig.inexact.arnoldi} reports our results for $\mathbf{v}=(1,\dots,1)^T / \sqrt{n}$
and the same matrices as in Example \ref{ex-arnoldi}: 
$A =\textrm{Toeplitz}(1,\underline{2}, 0.1, -1) \in {\cal B}_{200}(2,1)$ (left), 
and the matrix {\tt pde225} from the Matrix Market repository \cite{matrixmarket} (right).
For constant accuracy $\epsilon_j = tol/m$ (dashed line), the solid
line shows the residual norm $||\mathbf{r}_j||$ as the iteration $j$ proceeds. 
For variable accuracy
$\epsilon_j = \bar \epsilon_j$ obtained from (\ref{eq.choice.accuracy}) (stars)
the circles display the residual norm $\|\mathbf{\bar r}_j\|$.
We set $tol = 10^{-10}$ and $\epsilon^{(m)} = 10^{-1}$.
The maximum approximation space dimension $m$ was chosen as
 the smallest value for which the bound (\ref{eq-arnoldi-prebound}) is lower than $tol$,
respectively $m=20$ and $m=31$.
The fields of values of the matrices can be obtained starting from those reported
   in the left plots of Figure~\ref{fig-sec3-arnoldi}, where however now
the original semi-axes $a,b$ of the elliptical sets considered 
   for the computation of $s_j$
   are increased by $\epsilon^{(m)}/b$ and $\epsilon^{(m)}$ respectively.
The plots show visually overlapping residual norm histories for the two choices of $\epsilon_j$, illustrating
that in practice no loss of information takes place during the relaxing strategy.
\end{exm}

\begin{figure}[tbp]
\centering
    \includegraphics[width = 0.49\textwidth]{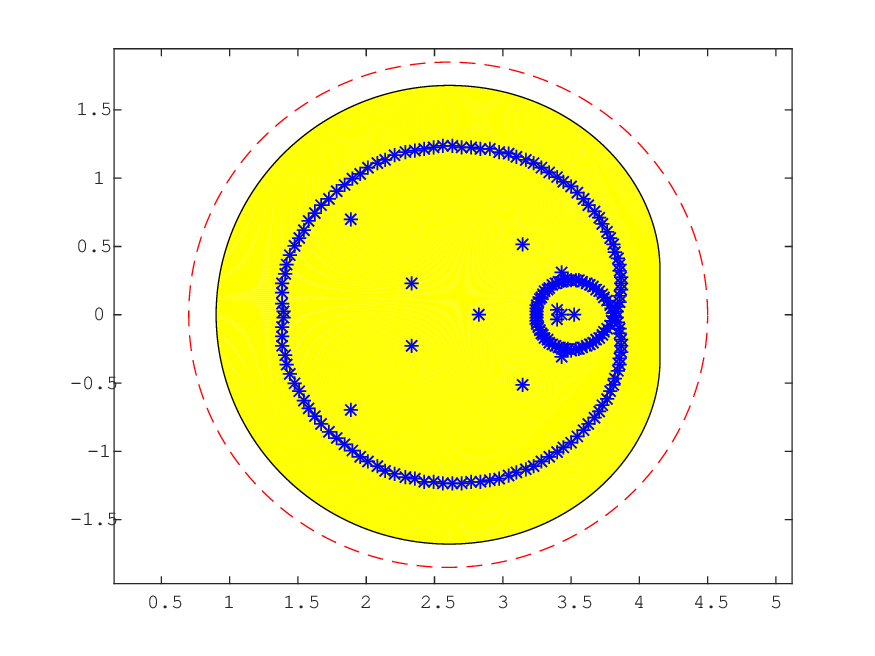}
    \includegraphics[width = 0.49\textwidth]{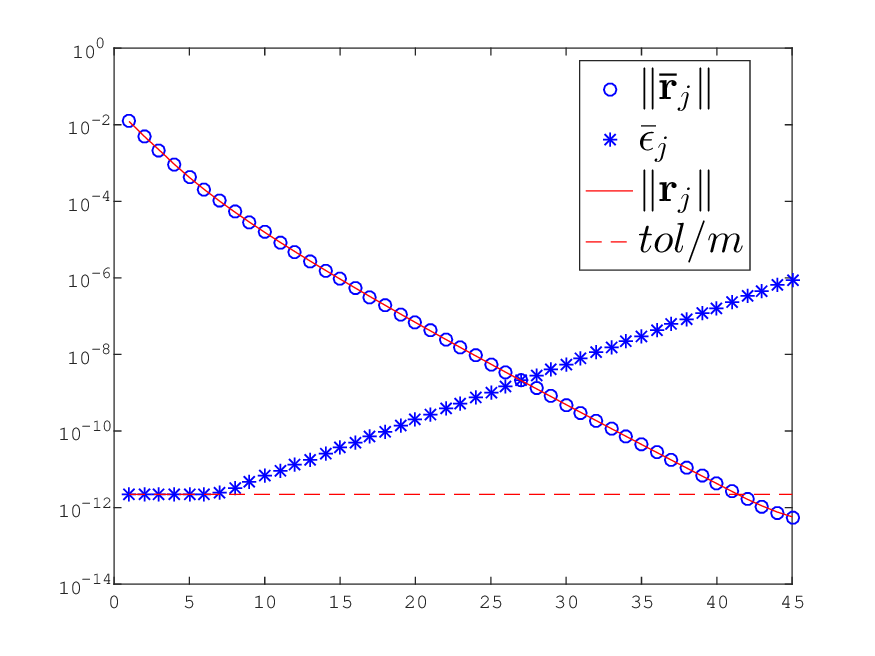}
\caption{Example \ref{ex.inex.expsqrt}.  Approximation of $e^{-A} \mathbf{v}$ 
with  $A=\textrm{Toeplitz}(-1,1,\underline{3},0.1) \in \mathcal{B}_{200}(1,2)$ and $\mathbf{v} = (1,\dots,1)^T / \sqrt{n}$.
Left: spectral information. Right:
Residual norm $||\mathbf{r}_j||$ with constant accuracy $\epsilon_j = tol/m$,
and residual norm $\|\mathbf{\bar r}_j\|$ with  $\epsilon_j = \bar \epsilon_j$ by (\ref{eq.choice.accuracy})
as the inexact Arnoldi method proceeds.
\label{fig.inexact.arnoldi.expsqrt}
}
\end{figure}

Consider the differential equation $ y^{(2)}= A y$, with $y(0) = \mathbf{v}$.
Its solution can be expressed as $y(t) = \exp(-t\sqrt{A}) \mathbf{v}$, and our results can be
applied to this case as well.
This time the upper bound $s_j$ for $|\mathbf{e}_m^T f(H_m) \mathbf{e}_1|$
is obtained from Corollary~\ref{cor_expsqrt}.
\begin{exm}\label{ex.inex.expsqrt}
For the same experimental
setting as in Example~\ref{ex-inex} we consider approximating $\exp(-\sqrt{A})\mathbf{v}$, 
for $A=\textrm{Toeplitz}(-1,1,\underline{3},0.1) \in \mathcal{B}_{200}(1,2)$,
$\mathbf{v}=(1,\dots,1)^T / \sqrt{200}$ and $m=35$.
Figure~\ref{fig.inexact.arnoldi.expsqrt} reports on our findings, with the same description
as for the previous example.
Here $s_j$ in (\ref{eq.choice.accuracy}) is obtained from Corollary~\ref{cor_expsqrt}, and it is
used to relax the accuracy $\epsilon_j$. Similar considerations apply.
\end{exm}

\section{Conclusions}\label{sec-conclusion}
Exploiting the described bounds for the off-diagonal decay pattern of functions of non-Hermitian banded matrices, 
we have derived bounds for the residual associated with the matrix function approximation given by the Arnoldi algorithm. 
As expected, the described bounds are influenced by the dependence between the predicted decay rate and the shape and dimension of the set enclosing the field of values of $A$. 
The closer $E$ is to the field of values, the sharper the bound.
We have also used the described decay estimates to define a strategy for setting the accuracy 
of the inexactness of matrix-vector products in Arnoldi approximations of matrix functions applied to a vector. 
Similar results can be obtained for other Krylov-type approximations whose projection and restriction matrix $H_m$ has a semi-banded structure. 
This is the case for instance of the Extended Krylov subspace approximation; see, e.g., \cite{KnizhnermanSimoncini2010} and references therein.

\begin{acknowledgements}
We are indebted with Leonid Knizhnerman for a careful reading of a earlier version of this manuscript, 
and for his many insightful remarks which led to great improvements of our results. 
We also thank Michele Benzi for several suggestions.
\end{acknowledgements}

\appendix

\section{Technical proofs}

\subsection*{Proof of corollary \ref{cor_exp}}
 Let $\rho = \sqrt{a^2 - b^2}$ be the distance between the foci and the center of the ellipse (i.e., the boundary of $E$), and let $R = (a + b)/\rho$.
 {Then a conformal map for $E$ is
 \begin{equation}\label{eq:phi:ell}
    \phi(w) = \frac{w - c - \sqrt{(w-c)^2 - \rho^2 }}{\rho R},
 \end{equation}
 and its inverse is}
 \begin{equation}\label{eqn-psi-ellipse}
    \psi(z) = \frac{\rho}{2}\left ( Rz + \frac{1}{Rz} \right) + c, 
 \end{equation}
 see, e.g., \cite[chapter II, Example 3]{suetin}.
 Notice that 
 $$ 
\max_{|z| = \tau} |e^{\psi(z)}| = \max_{|z| = \tau} e^{\Re(\psi(z))} = 
e^{\frac{\rho}{2}\left ( R\tau + \frac{1}{R\tau} \right) + c_1}.
$$
 Hence by Theorem \ref{thm-decay-level-set} we get
    $$ \left| \left ( e^{A} \right )_{k,\ell} \right|	\leq  
			2 \frac{\tau}{\tau - 1} e^{c_1} e^{\frac{\rho}{2}\left ( R\tau + \frac{1}{R\tau} \right)} \left( \frac{1}{\tau} \right)^\xi.$$
The optimal value of $\tau > 1$ that minimizes
 $ e^{\frac{\rho}{2}\left ( R\tau + \frac{1}{R\tau} \right)} \left( \frac{1}{\tau} \right)^\xi$ is
$$\tau = \frac{\xi + \sqrt{\xi^2 + \rho^2 }}{\rho R}.$$
Moreover the condition $\tau >1$ is satisfied if and only if $\xi > \frac{\rho}{2}\left(R - \frac{1}{R} \right) = b$. 
Finally, noticing that
$$ \psi\left(\frac{\xi + \sqrt{\xi^2 + \rho^2 }}{\rho R}\right) - c_1 = \frac{1}{2}\left (\xi + \sqrt{\xi^2 + \rho^2 } + \frac{\rho^2}{\xi + \sqrt{\xi^2 + \rho^2 }} \right) = \xi q(\xi),  $$
and collecting $\xi$  the proof is completed. \qed


\subsection*{Proof of corollary \ref{cor_expsqrt}}

 The function $f(z) = \exp(- \sqrt{z})$ is analytic on $\mathbb{C} \setminus (-\infty,0)$.
 Since we consider the principal square root, then $\Re(\sqrt{z})\geq0$,
 and so $$|\exp(- \sqrt{z})| = \exp(-\Re(\sqrt{z})) \leq 1.$$
 Hence, by Theorem \ref{thm-decay-level-set} we can determine $\tau$ for which
$$ \left| \left ( e^{-\sqrt{A}} \right )_{k,\ell} \right|	\leq  
			 2 \frac{\tau}{\tau - 1} \left ( \frac{1}{\tau} \right)^\xi.$$
For every $\varepsilon > 0$ close enough to zero, we set the parameter
$$\tau_\varepsilon = |\phi(\varepsilon)| = \left |\frac{c - \varepsilon + \sqrt{(c-\varepsilon)^2 -\rho^2}}{\rho R} \right|, $$
with $\phi(w)$ as in \eqref{eq:phi:ell} and $\psi(z)$ its inverse \eqref{eqn-psi-ellipse}.
Then the ellipse $\{ \psi(z), \, |z| = \tau_\varepsilon  \}$ is contained in $\mathbb{C} \setminus (-\infty,0]$.
Letting $\varepsilon \rightarrow 0$ concludes the proof. \qed

\bibliographystyle{spmpsci}      
\bibliography{Manuscript}   


\end{document}